\documentclass[12pt]{article}

\usepackage[english]{babel} 

\usepackage{amsmath,amssymb,amsbsy,amsfonts,amsthm,
  latexsym,amsopn,amstext,amsxtra,euscript,amscd}

\begin{document} \newtheorem{prop}{Proposition}[section] \newtheorem{lem}[prop]{Lemma} \newtheorem{cor}[prop]{Corollary} 
\newtheorem{conj}[prop]{Conjecture} \newtheorem{defi}[prop]{Definition} \newtheorem{thm}[prop]{Theorem} \newtheorem{rem}[prop]{Remark} 
\newtheorem{rems}[prop]{Remarks} \newtheorem{fac}[prop]{Property} \newtheorem{facs}[prop]{Properties} \newtheorem{com}[prop]{Comments} 
\newtheorem{prob}{Problem} \newtheorem{problem}[prob]{Problem} \newtheorem{ques}{Question} \newtheorem{question}[ques]{Question} 
\newtheorem{ejem}[prop]{Example}


\def\scr{\scriptstyle} \def\\{\cr} \def\({\left(} \def\){\right)} \def\[{\left[} \def\]{\right]} \def\<{\langle} \def\>{\rangle} 
\def\fl#1{\left\lfloor#1\right\rfloor} \def\rf#1{\left\lceil#1\right\rceil} \def\defn{\noindent{\bf Definition\/}\ \ } \def\rem{\noindent{\bf 
Remarks\/}\ \ } \def\card#1{\vphantom{#1}^{\#}}

\def\Z{\mathbb Z}
\def\C{\mathcal C}
\def\N{\mathbb N }
\def\P{\mathbb P} \def\R{\mathbb R} \def\N{\mathbb N} \def\A{\mathcal A} \def\cS{\mathcal S} \def\e{{\bf e}} \def\cR{\mathcal R} \def\cP{\mathcal P} 
\def\cQ{\mathcal Q} \def\cT{\mathcal T} \def\p{\textbf p} \def\q{\textbf q} \def\r{\textbf r} \def\s{\textbf s} \def\rad{{\mathrm{rad}\/}} 
\def\eps{\varepsilon} \def\ep{{\mathbf{\,e}}_p} \def\epp{{\mathbf{\,e}}_{p-1}}

\def\xxx{\vskip5pt\hrule\vskip5pt} \def\yyy{\vskip5pt\hrule\vskip2pt\hrule\vskip5pt}

\newcommand{\comm}[1]{\marginpar {\fbox{#1}}} \newcommand{\1}{1\!{\mathrm l}} \newcommand{\croix}{ set\frown}






\def\bbbr{{\rm I\!R}} 
\def\bbbf{{\rm I\!F}} \def\bbbh{{\rm I\!H}} \def\bbbk{{\rm I\!K}} \def\bbbl{{\rm I\!L}} \def\bbbp{{\rm I\!P}} \newcommand{\lcm}{{\rm lcm}} 
\def\bbbone{{\mathchoice {\rm 1\mskip-4mu l} {\rm 1\mskip-4mu l} {\rm 1\mskip-4.5mu l} {\rm 1\mskip-5mu l}}} \def\bbbc{{\mathchoice 
{\setbox0=\hbox{$\displaystyle\rm C$}\hbox{\hbox to0pt{\kern0.4\wd0\vrule height0.9\ht0\hss}\box0}} {\setbox0=\hbox{$\textstyle\rm C$}\hbox{\hbox 
to0pt{\kern0.4\wd0\vrule height0.9\ht0\hss}\box0}} {\setbox0=\hbox{$\scriptstyle\rm C$}\hbox{\hbox to0pt{\kern0.4\wd0\vrule height0.9\ht0\hss}\box0}} 
{\setbox0=\hbox{$\scriptscriptstyle\rm C$}\hbox{\hbox to0pt{\kern0.4\wd0\vrule height0.9\ht0\hss}\box0}}}} \def\bbbq{{\mathchoice 
{\setbox0=\hbox{$\displaystyle\rm Q$}\hbox{\raise 0.15\ht0\hbox to0pt{\kern0.4\wd0\vrule height0.8\ht0\hss}\box0}} {\setbox0=\hbox{$\textstyle\rm 
Q$}\hbox{\raise 0.15\ht0\hbox to0pt{\kern0.4\wd0\vrule height0.8\ht0\hss}\box0}} {\setbox0=\hbox{$\scriptstyle\rm Q$}\hbox{\raise 0.15\ht0\hbox 
to0pt{\kern0.4\wd0\vrule height0.7\ht0\hss}\box0}} {\setbox0=\hbox{$\scriptscriptstyle\rm Q$}\hbox{\raise 0.15\ht0\hbox to0pt{\kern0.4\wd0\vrule 
height0.7\ht0\hss}\box0}}}} \def\bbbt{{\mathchoice {\setbox0=\hbox{$\displaystyle\rm T$}\hbox{\hbox to0pt{\kern0.3\wd0\vrule 
height0.9\ht0\hss}\box0}} {\setbox0=\hbox{$\textstyle\rm T$}\hbox{\hbox to0pt{\kern0.3\wd0\vrule height0.9\ht0\hss}\box0}} 
{\setbox0=\hbox{$\scriptstyle\rm T$}\hbox{\hbox to0pt{\kern0.3\wd0\vrule height0.9\ht0\hss}\box0}} {\setbox0=\hbox{$\scriptscriptstyle\rm 
T$}\hbox{\hbox to0pt{\kern0.3\wd0\vrule height0.9\ht0\hss}\box0}}}} \def\bbbs{{\mathchoice {\setbox0=\hbox{$\displaystyle \rm 
S$}\hbox{\raise0.5\ht0\hbox to0pt{\kern0.35\wd0\vrule height0.45\ht0\hss}\hbox to0pt{\kern0.55\wd0\vrule height0.5\ht0\hss}\box0}} 
{\setbox0=\hbox{$\textstyle \rm S$}\hbox{\raise0.5\ht0\hbox to0pt{\kern0.35\wd0\vrule height0.45\ht0\hss}\hbox to0pt{\kern0.55\wd0\vrule 
height0.5\ht0\hss}\box0}} {\setbox0=\hbox{$\scriptstyle \rm S$}\hbox{\raise0.5\ht0\hbox to0pt{\kern0.35\wd0\vrule 
height0.45\ht0\hss}\raise0.05\ht0\hbox to0pt{\kern0.5\wd0\vrule height0.45\ht0\hss}\box0}} {\setbox0=\hbox{$\scriptscriptstyle\rm 
S$}\hbox{\raise0.5\ht0\hbox to0pt{\kern0.4\wd0\vrule height0.45\ht0\hss}\raise0.05\ht0\hbox to0pt{\kern0.55\wd0\vrule height0.45\ht0\hss}\box0}}}} 
\def\bbbz{{\mathchoice {\hbox{$\sf\textstyle Z\kern-0.4em Z$}} {\hbox{$\sf\textstyle Z\kern-0.4em Z$}} {\hbox{$\sf\scriptstyle Z\kern-0.3em Z$}} 
{\hbox{$\sf\scriptscriptstyle Z\kern-0.2em Z$}}}} \def\ts{\thinspace}

\def\squareforqed{\hbox{\rlap{$\sqcap$}$\sqcup$}} \def\qed{\ifmmode\squareforqed\else{\unskip\nobreak\hfil 
\penalty50\hskip1em\null\nobreak\hfil\squareforqed \parfillskip=0pt\finalhyphendemerits=0\endgraf}\fi}

\def\cD{{\mathcal D}} \def\cE{{\mathcal E}} \def\cF{{\mathcal F}} \def\cG{{\mathcal G}} \def\cH{{\mathcal H}} \def\cI{{\mathcal I}} \def\cJ{{\mathcal 
J}} \def\cK{{\mathcal K}} \def\cL{{\mathcal L}} \def\cM{{\mathcal M}} \def\cN{{\mathcal N}} \def\cO{{\mathcal O}} \def\cP{{\mathcal P}} 
\def\cQ{{\mathcal Q}} \def\cR{{\mathcal R}} \def\cS{{\mathcal S}} \def\cT{{\mathcal T}} \def\cU{{\mathcal U}} \def\cV{{\mathcal V}} \def\cW{{\mathcal 
W}} \def\cX{{\mathcal X}} \def\cY{{\mathcal Y}} \def\cLi{{\rm Li}} \def\cZ{{\mathcal Z}}

\def\Re{\mathrm{Re}}
\def\Im{\mathrm {Im}}
\def \Prob{{\mathrm {}}}

\title{A remark on Ruzsa's construction of an infinite Sidon set}

\author{\sc {Juan Pablo Maldonado L\'opez} \\{ITWM Fraunhofer}\\ {Fraunhofer Platz 1}\\ {67663 Kaiserslautern, Germany}\\ {\tt maldonad@itwm.fraunhofer.de} }
\maketitle

\footnote{AMS Classification Subject: 11P21,11B75. Keywords: Number theory, Additive number theory, Sidon sets, Sum-free sets}

\begin{abstract}
A Sidon set is a subset of the integers with the property that the sums of every two elements are distinct. In 1998, I.Ruzsa gave a probabilistic construction of an infinite Sidon set whose counting function is given by $x^{\sqrt{2}-1+o(1)} $. In this work, we explain the details of the simplification of Ruzsa's construction suggested in \cite{RC}.
\end{abstract}

\section{Introduction}
 A set of integers $\mathcal{S}$ is called a $\textit{Sidon set}$ or $\textit{Sidon sequence} $ if the sums of every two elements of $\mathcal{S}$ are distinct. Observe that writing \textit{differences} instead of \textit{sums} gives an equivalent definition. For example, the set $\{1,2,4,8,16, \ldots \} $ is an infinite Sidon set. Those sets arose in the 30's in the context of 
Fourier analysis. The hungarian analyst Simon Sidon asked Paul Erd\H{o}s about the size of those sets and since then they become of particular interest to number theorists.  

An interesting problem is to construct large Sidon sets. One can ask for either the larger Sidon set contained in the set 
$\{1,2, \ldots ,n\} $ or the asymptotic behaviour of an infinite Sidon set. For the finite case, a Sidon set may be constructed by the greedy algorithm. For the infinite case, it was proven by Erd\H{o}s that the number of elements in a Sidon set up to $x$ is  $\sim \sqrt{x}.$ The current record holder is the hungarian 
mathematician Imre Ruzsa, who built a Sidon set of size $x^{\sqrt{2}-1+o(1)}$ in \cite{R}. Ruzsa's proof is based on the fact that the prime numbers are a multiplicative Sidon set, so the set $\{\log p : p \ \ \text{prime}\}$ is an additive Sidon set. However, this set is unbounded, which created some trouble. Following Ruzsa's construction, we obtain a Sidon set with the same asymptotic behaviour on a slightly different way, by considering the set of arguments of the Gaussian primes instead of the logarithms of the primes. Whenever possible, we have kept Ruzsa's notation, so that the reader familiar with Ruzsa's paper can easily appreciate differences and similarities. It is, however, not required to read Ruzsa's paper in advance.

On the following section, we discuss the finite case and then in the following two sections we explain the construction on the infinite case. The construction is probabilistic, so we will first go through the combinatorial and number theoretical details and in the last section we will go through the probabilistic argument. It is worth mentioning that the 
analog for 
three-sums, that is, a \textit{large} infinite set with the property that every sum of three elements is different (by \textit{large} meaning larger than the obtained from the greedy algorithm), has not yet been constructed.

\section{Finite Sidon sets}
We will first show how can we obtain finite Sidon sets. The first natural approach is the greedy algorithm. Let $a_1=1$ and for $k > 1$ take $a_k$ such that $a_k \notin \{a_i+a_j-a_l | 1 \leq  i,j,l \leq k-1  \}.$ We immediately see that $a_k \leq (k-1)^3+1 $ since we have at most $(k-1)^3 $ forbidden choices for $\{i,j,l\} $ and so with this construction we can get a Sidon set of size $\sim n^\frac{1}{3} $ contained in the set $\{1,2, \ldots, n \}.$

The well known fact that integers can be written in an essentially unique way as product of prime numbers is remarkably useful in the construction of a large Sidon set. We have to translate this multiplicative property into an additive one, the natural way being through the logarithm function. 

As usual, denote with $[y]$ the largest integer less or equal than $y$ and $\{y\}=y-[y]$.

\begin{thm}
The set
\begin{equation*}
\mathcal{X}:= \left \{x_p \in \N :  x_p = \left [\frac{2n}{\log n}\log p \right ] \  \ p \leq \sqrt{\frac{n}{2\log n}},\  p \ \text{prime.} \right  \} 
\end{equation*}
\noindent
is a Sidon set of cardinality $\sim \frac{\sqrt{2n}}{\log^{3/2}n}$ for large enough $n$. 
\end{thm}

\begin{proof}
Our first claim is that this set is in fact a Sidon set. To see this, suppose that we have $p,q,r,s$ with $\{p,q\}\neq \{r,s\} $ such that

$$x_p+x_q=x_r+x_s.$$

\noindent
Without loss of generality, assume that $pq>rs.$ Observe that since we have that $x_p+x_q-x_r-x_s=0,$

\begin{equation*}
\frac{2n}{\log n }( \log p + \log q - \log r - \log s)  =  \Big \{\frac{2n}{\log n}\log p \Big \}  + \Big \{ \frac{2n}{\log n} \log q\Big \} - \Big \{ \frac{2n}{\log n} \log r\Big\} - \Big \{\frac{2n}{\log n} \log s \Big \}
\end{equation*}
\noindent
by using the fact that, for $x,y,z,w$ real numbers, the inequality 

\begin{equation}
\label{ineq:floor}
|\{x\}+\{y\}-\{z\}-\{w\} | \leq 2
\end{equation}
 holds, we have 

$$ \frac{2n}{\log n}\log \frac{pq}{rs} \leq 2 $$

\noindent
from which follows that
 $$ \frac{\log n}{n} \geq \log \frac{pq}{rs} $$ 
\noindent
but since

\begin{eqnarray*}
\log \frac{pq}{rs}& = &  \log \Big( 1+\frac{pq-rs}{rs}\Big) \\
& \geq & \log \Big(1 + \frac{1}{rs}\Big) \\
& \geq & \frac{1}{2rs} \\
&>& \frac{\log n}{n}
\end{eqnarray*}
where the first inequality follows by unique factorization, the second inequality from the inequality $(1+x)^2\geq e^x$ for $x<2$ and by taking logarithms on both sides, and the third inequality follows from the definition of $r$ and $s$.
\end{proof}

Ruzsa, in a very clever way, took advantage of the previous theorem to create an infinite Sidon set by considering the sequence of the logarithms of the primes. However, some difficulties arose from the fact that this sequence is unbounded. We now introduce another way of constructing a finite Sidon set, which will help us to motivate an analog construction for the infinite case. 

Let $\P$ be the set of prime numbers congruent to $1$ modulo $4$. For $p \in \P$, take a Gaussian prime $\rho_p$ such that $p =\rho_p \bar{\rho_p} $, with $\bar{\rho_p}$ such that $\Re{\bar{\rho_p}}>\Im{\bar{\rho_p}}>0 $, where $\Re z, \Im z$ denote the real and imaginary parts of the complex number $z$. 
Write $\frac{\bar{\rho_p}}{\rho_p} = e^{2 \pi i \phi_p} $ with $\phi_p$ a number in $[0,1).$ We write $|x|$ for the absolute value of the real number $x$ and $\|z\|$ for the norm of the complex number $z$. Note that the sequence of numbers $(\phi_p)_{p \in\P} $ is a 
Sidon set, for if we have $\phi_p+\phi_q=\phi_r+\phi_s $, that would imply that $\bar{\rho_p}\bar{\rho_q}\rho_r\rho_s=\rho_p\rho_q\bar{\rho_r}\bar{\rho_s} $ which 
is impossible, since Gaussian integers also have the unique factorization property. So if the $\phi_p$ were integers, we would have a Sidon set; since they are not, we would have to work them out to build a Sidon set, as we did in the previous theorem. We now state this precisely.

\begin{thm}
The set
$$\mathcal C := \Big\{c_p \in \N :  c_p = \[n \phi_p \], p \in \P,\  p \leq \frac{\sqrt{n}}{4} \Big\}, $$
\noindent
is a Sidon set contained in the set $\{1,2\ldots n\}$ with $\frac{\sqrt{n}}{4 \log n}$ elements.
\end{thm}

\begin{proof}
Suppose that we have four elements in this set, such that

$$c_p + c_q = c_r + c_s $$

\noindent
with $\{p,q\}\neq \{r,s\}$ and $pq > rs. $ Consider

$$ n (\phi_p+\phi_q-\phi_r-\phi_s)= \{n\phi_p\}+\{n\phi_q\}-\{n\phi_r\}-\{n\phi_s\}$$
\noindent
Observe that

\begin{eqnarray*}
\Big\| \frac{\bar{\rho_p}\bar{\rho_q}}{\rho_p\rho_q} - \frac{\bar{\rho_r}\bar{\rho_s}}{\rho_r \rho_s} \Big\| & = &
\Big\|1-\frac{\rho_p\rho_q\bar{\rho_r}\bar{\rho_s}}{\bar{\rho_p}\bar{\rho_q}\rho_r\rho_s} \Big\| \\
&= & \Big\|1-e^{2\pi i (\phi_p+\phi_q-\phi_r-\phi_s)} \Big\| \\
& \leq & 2\pi |\phi_p+\phi_q-\phi_r-\phi_s | \\
& \leq & \frac{4\pi}{n}
\end{eqnarray*}
the first inequality being easily obtained by geometric interpretation and the second inequality follows from \eqref{ineq:floor}. On the other hand we get

\begin{eqnarray*}
\Big\| \frac{\bar{\rho_p}\bar{\rho_q}}{\rho_p\rho_q} - \frac{\bar{\rho_r}\bar{\rho_s}}{\rho_r \rho_s} \Big\| &= &
\Big\| \frac{\bar{\rho_p}\bar{\rho_q}\rho_r\rho_s-\bar{\rho_r}\bar{\rho_s}\rho_p\rho_q}{\rho_p\rho_q\rho_r\rho_s}\Big\|\\
&\geq & \frac{1}{\sqrt{pqrs}} \\
&\geq & \frac{16}{n}
\end{eqnarray*}
\noindent
from the above inequalities follows that

$$\frac{16}{n}\leq \frac{4\pi}{n} $$

\noindent
which can not hold for positive $n.$ By the prime number theorem, we have that the cardinality of $\mathcal{C} $ is
as claimed.
\end{proof}

\section{The construction}

Keeping the notation from the previous section, we will now construct an infinite Sidon set. As Ruzsa did with the logarithms of the primes, we will use the $\phi_p$'s; we will take 
their binary expansion and, since this expansion might be infinite, we will have to cut 
it somewhere and then use the digits 
to construct a Sidon sequence. Before we start to play with the digits, we will add a parameter $\alpha \in [1,2) $ and consider the set $\{ \alpha \phi_p \in \R \ : \ p \in \P \}. $ The 
values of $ \alpha$ such that the resulting set $\mathcal A_\alpha$ (after truncating and rearranging digits) is a Sidon set will be shown to be a set of positive 
measure. From now on, we 
will forget the rest of the prime numbers and we will just care about those prime numbers that belong to the set $\P = \{ p \cong 1 \mod 4 \  : \ p \ \  \text{prime} \}$.

Let $\beta$ be a fixed positive real number and consider the integer $K_p > 2$ such that

$$ 2^{(K_p-2)^2} < p^{\beta} < 2^{(K_p-1)^2} $$

\noindent  we also consider

$$P_K = \{p \in \P : K_p = K\} .$$ \noindent that is, those primes whose $\beta$ power is more or less of the same size.

For $ p \in P_K$ and $\alpha \in [1,2) $ take the number

$$ m_p:=  \big[ 2^{K^2} \alpha \phi_p \big] = \sum_{i=1}^{K^2}\delta_{ip}2^{K^2-i} $$

\noindent with $\delta_{ip} \in \{0,1\}$. These numbers, while $p$ runs over $\P$, are 
the main ingredient for our Sidon sequence. We have just cut the binary expansion of the numbers $\alpha \phi_p$ in a certain place that depends on $p.$
\noindent
We cut this number into $ \Delta_{1p}, \Delta_{2p}, \ldots, \Delta_{Kp} $ pieces such that

$$ \Delta_{ip} = \sum^{i^2}_{j=(i-1)^2+1} \delta_{jp}2^{i^2-j}$$

\noindent
and so we have that

\begin{eqnarray}
\label{eqn:dip}
\Delta_{ip} \leq \sum_{j=(i-1)^2+1}^{i^2}2^{i^2-j} = 2^{2i}-1.
\end{eqnarray}

\noindent Now we will rearrange these blocks. Informally, we put the blocks $1$ to $K$, the first block corresponding to the first digit; the second block, to the next 
four digits; the third block, to the the next nine, and so on up to the last $K^2$ from right to left and leave three spaces between consecutive blocks.  We also write $1$ 
on the second space from right to left from the $K-$th block. This $1$, who is the first non-zero digit of our new number, is used to have precise 
information on the size of our new number. For example, if we had the sequence $0.10101010111010$ after cutting in consecutive blocks we get

\begin{equation*}
 \Delta_1 = 1, \Delta_2=0101, \Delta_3= 010111010
\end{equation*}
 \noindent and after the three zeros and the $1$ we get

$$\textbf{10}010111010\textbf{000}0101\textbf{000}1$$

 \noindent where the numbers in bold letters correspond to the digits we insert between the consecutive blocks. More precisely, we have the 
number

$$a_p := \sum_{i=1}^{K_p} \Delta_{ip}2^{(i-1)^2+3i} + 2^{K_p^2 + 3K_p +2}. $$ \noindent From the choice of $K$, since $2^{K_p^2+3K_p+2} < a_p < 2^{K_p^2+3K_p+3}$ we see that 
$a_p = p^{\beta + o(1)}.$ We also introduce the notation $t_p := 2^{K_p^2+3K_p+2}$. Two questions may arise: what is the $t_p$ for? what about those 
zeros? 

Consider the identity in binary numbers

$$1 000\textbf{0}11 + 100\textbf{0}10 = 111\textbf{0}11 + 101 \textbf{0}10 $$

\noindent
observe that all these four numbers have a block of three zeros in the same position. When you add these numbers, observe that the block of zeros prevents you from 'carrying' ones to the other blocks, so that in some sense the blocks $1000, 1 00, 111, 101$ corresponding to the last four digits of each number (from right to left) contribute to the sum on each side of the equation in complete independence of the numbers living on the other side of the zeros block. So we must have that $1000 + 1 00 = 111 + 101 $, and $11 + 10 = 11 + 10 $ which is true.

How will this simple observation help us to construct our Sidon set? Consider the set $\A_{\alpha},$ for a fixed choice of $\alpha.$ Let $a_p,a_q,a_r,a_s \in \A_{\alpha} $ with $p,q,r,s \in \P.$ If we have that
\begin{equation}
 \label{eqn:sum}
a_p + a_q = a_r + a_s \end{equation}
\noindent
under the additional hyphotesis (which holds up to renaming of variables)

\begin{equation} \label{eqn:ineq} a_p > a_r \geq a_s > a_q. \end{equation} 
\noindent
we call the $4$-tuple $(p,q,r,s) \in \P^4$ a \textit{bad} $4$-tuple. Whenever we have a bad $4$-tuple we can remove the $a_i$ corresponding to the largest element of the $4$-tuple, and the remaining elements in $\A_{\alpha}$ (after removing all the largest elements from every bad $4$-tuples) now form a Sidon set. We are interested in an estimate of the number of such bad $4$-tuples. The least we have to remove, the larger our Sidon set will be.

The very particular way of constructing the elements of the set $\A_{\alpha} $ is what gives the key for counting the number of bad $4$-tuples. The extra zeros inserted between two consecutive blocks ensure that, as in the previous example, the sums are \textit{blockwise independent}. We make this statement precise in the following lemma.

\begin{lem} $(p,q,r,s)$ is a bad $4$-tuple iff $\Delta_{ip}+\Delta_{iq}=\Delta_{ir}+\Delta_{is} $ for all $i$ and $t_p + t_q = t_r + t_s. $ \end{lem} 

\begin{proof}

We assume that \eqref{eqn:sum} and \eqref{eqn:ineq} hold (the converse is immediate from the definition of the $a_p$'s. So suppose that 
$$a_p+a_q=a_r+a_s $$
\noindent
now observe that since
$$2^{K_p^2+3K_p+2} > \sum_{i=1}^{K_p}\Delta_{ip}2^{(i-1)^2+3i} $$
\noindent
and similarly for $q,r,s$, the contribution of the main parts $t_p,t_q,t_r,t_s$ is independent of the remaining digits of $a_p,a_q,a_r,a_s$ respectively, gives
$$t_p+t_q=t_r+t_s $$
\noindent
and from \eqref{eqn:ineq} and the uniqueness of the representation of integers in base $2$ follows that there exist $K$ and $L$ such that $K_p = K_r = K $ and $K_q=K_s=L$ with $K \geq L$. We still have to say something about
\begin{equation*}
\sum_{i=1}^{K} \Delta_{ip}2^{(i-1)^2+3i}+ \sum_{i=1}^{L} \Delta_{iq}2^{(i-1)^2+3i} =   \sum_{i=1}^{K} \Delta_{ir}2^{(i-1)^2+3i}+ \sum_{i=1}^{L} \Delta_{is} 2^{(i-1)^2+3i} 
\end{equation*}
\noindent
but since $$2^{i^2+3(i+1)} > \sum_{j=1}^{i}\Delta_{ip}2^{(j-1)^2+3j} $$
\noindent 
and similarly for $q,r,s$, we see that, for $i<L$, 

\begin{eqnarray*}
\sum_{j=1}^i(\Delta_{jp}+\Delta_{jq})2^{(j-1)^2+3j} = \sum_{j=1}^i (\Delta_{jr}+\Delta_{js})2^{(j-1)^2+3j}
\end{eqnarray*}
and since the terms in parenthesis do not affect the other summands because their total sum is $\leq 2^{2j+1}-2 $ by \eqref{eqn:dip}, we must have that

$$ \Delta_{ip}+\Delta_{iq} = \Delta_{ir}+\Delta_{is}$$

\noindent
since, for $i>L$ we have that $\Delta_{iq}, \Delta_{is} = 0,$ the conclussion follows.

\end{proof}

The relevant difference between Ruzsa's approach and ours is the fact that Ruzsa had to solve the extra problem of inserting two more spaces and use one of those spaces to store the digits corresponding to $[\log p].$ 

While proving the above lemma, we showed an useful criterion in terms of the $t_p$'s. This will help us to count the number of solutions of \eqref{eqn:sum}, i.e. the bad $4$-tuples. For the following lemma, recall that $m_p= \Big[2^{K^2}\alpha\phi_p\Big] $

\begin{lem} If \eqref{eqn:sum} and \eqref{eqn:ineq} hold, there exist $K,L$ such that $p,r \in P_K$, $q,s \in P_L $, $K \geq L$ and

 $$m_p+m_q=m_r+m_s  $$ \end{lem}

\begin{proof}
The first assertion follows from the previous lemma. The second assertion is immediate from the corresponding identity for the blocks which was also proved in the previous lemma.
\end{proof}

We will find necessary conditions on the bad $4-$tuples $(p,q,r,s)$. The following lemma essentially follows from the fact that $\phi_p + \phi_q = \phi_{pq}$.

\begin{lem} If $(p,q,r,s)$ is a bad $4$-tuple, with $K$ and $L$ as above, then

 \begin{eqnarray}
\label{eqn:nes}
|\phi_{pq}-\phi_{rs}| & < & 4\cdot2^{-L^2}  \\
(K-1)^2 +(L-1)^2 & > & \beta (L-1)^2
 \end{eqnarray}
 \end{lem} \begin{proof} Let $\rho_p, \rho_q, \rho_r, \rho_s$ the Gaussian primes with norms $\sqrt{p},\sqrt{q},\sqrt{r},\sqrt{s}$ 
respectively. Since $e^{2\pi i \phi_j} = \frac{\bar{\rho_j}}{\rho_j}$, we have

\begin{eqnarray*}
\Big \|\frac{\bar{\rho_p}\bar{\rho_q}}{\rho_p\rho_q} - \frac{\bar{\rho_r}\bar{\rho_s}}{\rho_r\rho_s}\Big\| &=& 
\Big \| \frac{\bar{\rho_q}\bar{\rho_q}\rho_r\rho_s-\bar{\rho_r}\bar{\rho_s}\rho_p\rho_q}{\rho_p\rho_q\rho_r\rho_s} \Big \| \\
&\geq& \frac{1}{\sqrt{pqrs}}
\end{eqnarray*}
\noindent
the inequality follows from the unique factorization property of Gaussian integers.

Since

\begin{eqnarray*}
\Big \| \frac{\bar{\rho_p}\bar{\rho_q}}{\rho_p\rho_q}-\frac{\bar{\rho_r}\bar{\rho_s}}{\rho_r \rho_s} \Big\| &=& 
\Big \| e^{2\pi i (\phi_p+\phi_q)}-e^{2\pi i(\phi_r+\phi_s)}  \Big \| \\
&=& \Big \|1-e^{2\pi i (\phi_p+\phi_q-\phi_r-\phi_s)}  \Big \| \\
&\leq& 2\pi |\phi_p+\phi_q-\phi_r-\phi_s| \\
&< & 8|\phi_p+\phi_q-\phi_r-\phi_s|
\end{eqnarray*}

\noindent
From the definition of the $ m_p$ and the triangle inequality follows that 

\begin{equation}
\label{eqn:log}
 \alpha |\phi_p + \phi_q -\phi_r - \phi_s| < |\alpha \phi_p - m_p| + 
|\alpha \phi_q - m_q|+|\alpha \phi_r - m_r|+|\alpha \phi_s - m_s| < 4\cdot 2^{-L^2}. 
\end{equation}

 \noindent Since $\alpha \geq 1,$ by combining the 
above inequalities we get

$$\frac{1}{\sqrt{pqrs}} < 32 \cdot 2^{-L^2} $$

\noindent
which implies
\begin{equation*} 2^{L^2-5}< \sqrt{pqrs} < 2^{\frac{(K-1)^2+(L-1)^2}{\beta}} \end{equation*} \noindent and the desired inequality follows for large enough $L$ from comparing the 
exponents. \end{proof}
\noindent
Observe that

$$|\phi_p+\phi_q-\phi_r-\phi_s| = |(\phi_p-\phi_r)-(\phi_s-\phi_q)|=|\phi_{p\bar{r}}-\phi_{s\bar{q}}| $$

\noindent
where $\phi_{p\bar{r}}$ denotes the argument of the complex number $\bar{\rho_p}\rho_r$ and 
$a_{s\bar{q}} $ is defined in an analogous way. For a given choice of $\rho_q\bar{\rho_s}$ we will count the corresponding points pairs $(p,r)$ such that \eqref{eqn:log} holds

\begin{lem}
Let $z_0=\rho_s\bar{\rho_q}.$ Consider a circle $\C $ with center $z_0$ and radius $R$. The number $n$ of lattice points $\{w_1,w_2,\ldots w_n\}$ on a circular sector of $\C$ with angle $\theta$ such that for $i=1,\ldots n$, $w_i=\rho_{p_i}\bar{\rho_{r_i}}$ for some $p_i,r_i \in P_K$ is at most $\theta R^2+1 $
\end{lem}

\begin{proof}
Consider the line segment joining $z_0$ and a point $w_i$. Observe that this segment does not contain a third point $w_j$, for if it did, we would have that the argument of $w_i$ equals the argument of $w_j$ and so

$$\phi_{p_i}-\phi_{r_i}=\phi_{p_j}-\phi_{r_j} $$

\noindent
which can not hold if $\{\phi_{p_i},\phi_{r_i}\} \neq \{ \phi_{p_j},\phi_{r_j}\}$ from our previous remark that the set of arguments of Gaussian primes is a Sidon set. So we can ennumerate our points in trigonometric sense. Now we consider the triangles with vertices $z_0$, $w_i$ and $w_{i+1}$  for $i=1,\ldots n-1.$ This is a set of disjoint triangles and the total area covered by them is less than the area of the circular sector, which is given by $\frac{\theta}{2}R^2.$ Since all the triangles have lattice points as vertices, we have that the area of each triangle is at most $\frac{1}{2}$ and since we have $n-1$ triangles we get

\begin{equation*}
\frac{n-1}{2} \leq \frac{\theta}{2} R^2 
\end{equation*}
\noindent
 and thus the desired inequality for $n$ follows.

\end{proof}

Consider the set

$$ \A_{KL} :=  \{ p,r \in P_K, q,s \in P_L, p\neq r, q\neq s : (p,q,r,s)\ \  \text{is} \  \text{bad}.\}$$
\noindent 
and let $|\A_{KL}| := A_{KL}.$ On the following lemma, we obtain an estimate for the number of bad $4$-tuples.

\begin{lem} The number of  bad $4$-tuples is

$$ A_{KL} \ll 2^{\frac{2}{\beta}((K-1)^2 + (L-1)^2 ) - L^2}$$

\end{lem}

\begin{proof}

For a fixed choice of $q,s$ we need to count the number of pairs $(p,r)$ with $p,r$ as above such that the inequality from the previous lemma occurs. Since $pr < 2^{\frac{2(K-1)^2}{\beta}}$ then we have that the norm of the lattice points we are interested in is less than $2^{\frac{(K-1)^2}{\beta}}.$ Take  $R= 2^{\frac{(K-1)^2}{\beta}}$ and let $\theta = 2^{- L^2}. $ So we have, from the previous lemma, that for given $q,s$, that the number of pairs $(p,r)$ we are interested in is at most $2^{\frac{2}{\beta}(K-1)^2-L^2}+1 \ll 2^{\frac{2}{\beta}(K-1)^2-L^2}$ since

$$ (K-1)^2 > (\beta-1)(L-1)^2 $$
\noindent
and so 

$$\frac{2}{\beta}(K-1)^2-L^2 > \frac{2}{\beta}(\beta-1)(L-1)^2 -L^2> 0 $$

\noindent
whenever $\frac{2}{\beta}(\beta-1)-1>0 $ i.e. $\beta>2 $ and $L$ is large enough.
\noindent
Since we have $2^{\frac{2}{\beta}(L-1)^2} $ possibilities for pairs $(q,s)$ we get that
$$A_{KL} < 2^{\frac{2}{\beta}((K-1)^2+(L-1)^2)-L^2}+ 2^{\frac{2}{\beta}(L-1)^2} $$
\noindent
and from the previous remark the inequality
$$A_{KL} \ll 2^{\frac{2}{\beta}((K-1)^2+(L-1)^2)-L^2} $$
\noindent
follows.
\end{proof}

\section{The probabilistic argument}

All the way up to here, the parameter $\alpha $ has not been relevant in the sense that all we have proved holds independently from the choice of $\alpha$. Up to now, we have been able to find a bound for the number of bad $4$-tuples, which will be useful. However, our approach does not suggest a way to refine this bound for particular values of $\alpha.$ The bound for the bad $4$-tuples we have showed in the previous section is not useful for small $L$. In order to obtain a better bound, we will first try to gather some information from the parameter $\alpha$.

\begin{lem} Let $(p,q,r,s)$ a bad $4$-tuple. Then we have

\begin{equation} \label{eqn:cong}
m_p \equiv m_r \mod 2^{K^2-L^2}. \end{equation}

\end{lem}

\begin{proof} We know that $\Delta_{ip} + \Delta_{iq} = \Delta_{ir} + \Delta_{is}.$ For $L < i< K$ we have that $ \Delta_{iq} = \Delta_{is} = 0,$ 
hence, $\Delta_{ip} = \Delta_{ir}. $ Recalling the construction of the blocks $\Delta_{ip} $ and $\Delta_{ir} $ from the digits of $m_p$ and 
$m_r $ we get that the corresponding digits on the binary expansion of these two numbers from the $L^2+1$-th position on (from right to left) should be equal. 
\end{proof}

Let $\mu $ be the Lebesgue measure of $\R.$ We now prove that we can avoid bad $4$-tuples by changing the parameter $\alpha.$

\begin{lem} Let $K > L$ be given and $p,r \in P_K$ such that there exists at least one pair $q,s \in P_L $ and an $\alpha$ such that \eqref{eqn:sum} holds. Then

$$ \mu \{ \alpha \in [1,2) : \eqref{eqn:cong} \ \text{holds}\} \ll 2^{L^2-K^2}.$$

\end{lem}

\begin{proof} Recall that $[x]-[y] = [x-y] + 0 $ or $1$. Then we have that

\begin{equation*}
 \Big[2^{K^2}\alpha (\phi_p-\phi_r) \Big] \equiv 0,1 \mod 2^{K^2-L^2}. \end{equation*} \noindent On a given interval of length $M:=2^{K^2-L^2} $ the set of 
$\alpha$ such that the previous congruence holds is an interval of length $2$. Let $N:=\big[ 2^{K^2} (\phi_p-\phi_r)\big].$ If we consider intervals of 
length $\frac{M}{N}$, the values of $\alpha$ such that \eqref{eqn:cong} holds lie on intervals of length $\frac{2}{N},$ one in each interval of length $\frac{M}{N}.$ The number of such intervals 
intersecting the interval $[1,2]$ is at most $1+\frac{N}{M}, $ hence

$$ \mu \{ \alpha \in [1,2) : \eqref{eqn:cong} \ \text{holds}\} \ll \frac{2}{N}\Big( 1+\frac{N}{M} \Big). $$

\noindent We have from \eqref{eqn:log} that 

$$\Big|\phi_p-\phi_r\Big| = \Big|\phi_s-\phi_q\Big| + O(2^{-L^2}). $$ \noindent Since

\begin{eqnarray*}
 \big|\phi_q-\phi_s\big|&  = &  \frac{1}{\pi}\Big\| \textit{Log} \frac{\bar{\rho_q}\rho_s}{\bar{\rho_s}\rho_q} \Big\| \\
&\geq&\frac{1}{2\pi} \Big\|1-\frac{\bar{\rho_q}\rho_s}{\rho_q\bar{\rho_s}}  \Big\| \\
&=& \frac{1}{2\pi}\Big\|\frac{\bar{\rho_q}\rho_s-\bar{\rho_s}\rho_q}{\rho_q\bar{\rho_s}}\Big\| \\
&\gg&  2^{-\frac{L^2}{\beta}} 
\end{eqnarray*}
 \noindent which follows from the fact that $q^{\beta}, s^{\beta}< 2^{(L-1)^2} $ and the inequality
$$\| e^z-1 \|\leq 2 \|z\| $$
\noindent
 for $\|z\|\leq 1.$ So we have that $$N \gg 2^{K^2-\frac{L^2}{\beta}} $$
\noindent
and thus

$$ N \gg 2^{K^2 - \frac{1}{\beta}L^2} > M. $$ 
\noindent Hence, $\frac{2}{B}\big( 1+\frac{N}{M} \big) \ll \frac{1}{N}\frac{N}{M} = \frac{1}{M} $ which 
concludes the proof of the lemma.

\end{proof}
\noindent

Observe that this new bound is good in the sense that not many $\alpha$'s contribute to complete bad $4$-tuples from a given pair $p,r$ whenever $L$ is small. In contrast, our previous bound for the number of bad $4$-tuples is not so good when $L$ is small, but it is good when $L$ is not far from $K$. This suggest that we may combine them somehow so that in average they tend to compensate and thus we can obtain a reasonable bound. Let $$T_{KL}(\alpha) = \# \{ p,q,r,s : p,r \in P_K, r,s, \in P_L, p\neq 
r, q \neq s \ a_p + a_q = a_r + a_s\}.$$

\begin{lem} For $K>L$ we have

$$ \int_1^2 T_{KL}(\alpha) d \alpha \ll 2^{\frac{2}{\beta}((K-1)^2+(L-1)^2)-K^2}.$$

\end{lem}

\begin{proof} Write $m=\mu \{ \alpha \in [1,2) : \eqref{eqn:sum} \ \text{holds} \}.$ Since $m=0$ whenever \eqref{eqn:nes} does not hold and $\ll 
2^{L^2-K^2}$ otherwise, summing for all possible values of $p,q,r,s$ we obtain $$\int_{1}^2 T_{KL}(\alpha) d\alpha \ll 2^{L^2-K^2}A_{KL}$$ \noindent 
from which the desired inequality follows by substituting the estimate for $A_{KL}$ given above. \end{proof} \noindent Define $T_K(\alpha) := \# \{ 
p,q,r,s : p,r \in P_K, \ \eqref{eqn:sum} \ \text{and} \ \eqref{eqn:ineq} \ \text{hold}\}. $ From the definition follows that $T_K(\alpha) = 
\sum_{L\geq K}T_{KL}(\alpha). $

\begin{lem}

The estimate

$$\int_1^2 T_K(\alpha) d\alpha \ll 2^{\frac{1}{\beta}(K-1)^2-2K} $$

\noindent holds. \end{lem}

\begin{proof} Since $T_{KL} (\alpha) \neq 0$ is possible only if $(K-1)^2 +(L-1)^2 > \beta (L-1)^2$,  if we define $\mathcal{L}$ to be the set of such $L$ we 
have that

\begin{eqnarray*} \int_1^2 T_K(\alpha) d\alpha & = & \sum_{L\leq K}\int_1^2 T_{KL}(\alpha) d\alpha \\ 
& = & \sum_{L \in \mathcal{L}} \int_1^2 T_{KL} 
(\alpha) d\alpha \\
 & \ll & 2^{\frac{2}{\beta}(K-1)^2-K^2}\sum_{L \in \mathcal{L}} 2^{\frac{2(L-1)^2}{\beta}} \\
 & \ll & 2^C2^{-2K} 
\end{eqnarray*}
\noindent
where

\begin{eqnarray*}
C &=&  \frac{2(K-1)^2}{\beta}-K^2+\frac{2(K-1)^2}{\beta(\beta-1)} \\
&=& \frac{2}{\beta-1}(K-1)^2-K^2 
\end{eqnarray*}
so we get that the main coefficient of the above expression is

$$\frac{2}{\beta-1}-1. $$
\noindent
We would like this coefficient to be equal to $\frac{1}{\beta}.$ In other words, we look for solutions of the equation

$$\beta^2-2\beta-1=0 $$

\noindent
The positive value of  $\beta$ that satisfies this last equation, for given $\delta$, is given by

$$\beta = 1+ \sqrt{2} $$
\noindent
 This also implies that $\beta < 3,$ from which the inequality for the linear term in the exponent of $2$ also follows.
\end{proof}

With all the machinery from the previous lemmas, we are ready to do the final step, which we resume in the following theorem.

\begin{thm}(Ruzsa, 1998) There exists an infinite Sidon set $\mathcal{S}$ such that the counting function $S(x) $ (i.e. the cardinality of $\mathcal{S}\cap [1,x $) satisfies 
$$S(x)=x^{\frac{1}{\beta}+o(1)}$$ 

\noindent
for $\beta$ as above,i.e. $\frac{1}{\beta}=\sqrt{2}-1 $.
\end{thm}

\begin{proof}

From the previous estimate, we get $$ \sum_{K} 2^{-(\frac{1}{\beta}(K-1)^2-K)} \int_1^2 T_K(\alpha) d\alpha \ll \sum_K 2^{-K}$$

\noindent it follows that

$$ \int_1^2\sum_K T_K(\alpha) 2^{-(\frac{1}{\beta}(K-1)^2-K)} d\alpha < +\infty$$

\noindent and hence for almost every $\alpha$, $T_K(\alpha) \ll 2^{\frac{1}{\beta}(K-1)^2-K} $ for large enough $K$, (depending on $\alpha$). From now on we 
focus on one of these $ \alpha.$ Denote with $\pi_1(x) $ the prime numbers smaller than $x$ that are congruent to $1$ modulo $4$. The cardinality of $P_K$ is given by Dirichlet's theorem on primes in arithmetic progressions \begin{equation*}
|P_K| = \pi_1 \big(2^{\frac{(K-1)^2}{\beta}} \big) - \pi_1 \big(2^{\frac{(K-2)^2}{\beta}}\big) \sim \frac{2^{\frac{(K-1)^2}{\beta}}}{2(K-1)^2\beta \log 2}
\end{equation*}

\noindent then, for large enough $K$, $T_K(\alpha) < \frac{|P_K|}{2}.$ This means that the bad $4$-tuples (i.e. the ones such that \eqref{eqn:sum} 
holds) are not so many; if we omit the $a_p$ from such $4$-tuples, the set of remaining elements has cardinality larger than $ \frac{|P_K|}{2}.$ If 
we denote with $ Q_K$ the set of remaining elements and take $\mathcal{S}$ to be the union of such $Q_K$ then $\mathcal{S}$ will be a Sidon set.

Let $S(x) $ be the counting function of $\mathcal{S}$. Since $a_p < 2^{K^2+3K+2} < 2^{(K+2)^2}$ for $K=\Big[ \sqrt{\log_2x}-2\Big]$ the set $Q_K$ 
consists of integers less than $x,$ from which follows that $ S(x) \gg \pi (2^{\frac{1}{\beta}K^2}) = x^{\frac{1}{\beta}+o(1)}$. Since we also 
have that

$$a_p > 2^{K^2+3K+1} > 2^{(K+1)^2} $$ \noindent taking $K=\Big[ \sqrt{\log_2 x}-1\Big] $ the set $Q_K$ has elements larger than $x$ and then $ S(x) 
\ll \pi_1(2^{\frac{1}{\beta}K^2}) = x^{\frac{1}{\beta}+o(1)}.$ From the previous estimates the theorem follows.

\end{proof}

It is worth emphasyzing that our construction is entirely based on Ruzsa's ideas. However, the technical simplifications help to a better understanding of Ruzsa's clever construction. The unique factorization of $\Z[i]$ is also a very important part of our work, as well as the geometrical insight one gets for free when working on $\Z[i].$

\noindent {\bf Acknowledgements.} This paper was written under the direction of Professor Javier Cilleruelo during the 2008 DocCourse in Additive Combinatorics held at the Centre de Recerca 
Matematica, Universitat Aut\'onoma de Barcelona. The author acknowledges the Centre de Recerca Matematica for its warm hospitality and Florian Luca for his careful proof reading.

\end{document}